\documentclass[10pt,letterpaper]{amsart}
\usepackage[parfill]{parskip}
\usepackage{amssymb}
\usepackage{amsmath}
\usepackage{amsthm}
\usepackage{amsfonts}

\usepackage[T1]{fontenc}

\usepackage{textcomp}
\usepackage{palatino}
\usepackage[text={6.5in, 9in}, centering]{geometry}

\usepackage[usenames]{color}

\newtheorem{theorem}{Theorem}
\newtheorem{lemma}[theorem]{Lemma}

\theoremstyle{definition}
\newtheorem{definition}[theorem]{Definition}
\theoremstyle{remark}
\newtheorem{remark}[theorem]{Remark}

\def\ocirc#1{\ifmmode\setbox0=\hbox{$#1$}\dimen0=\ht0
    \advance\dimen0 by1pt\rlap{\hbox to\wd0{\hss\raise\dimen0
    \hbox{\hskip.2em$\scriptscriptstyle\circ$}\hss}}#1\else
    {\accent"17 #1}\fi}

\newcommand{\R}{\mathbb{R}}

\newcommand{\N}{\mathbb{N}}

\newcommand{\supp}{\operatorname{supp}}

\newcommand{\sign}{\operatorname{sign}}
\newcommand{\Log}{\log}
\title[Blow-up scenarios for 3D NSE.]{Blow-up scenarios for 3D NSE exhibiting
sub-criticality with respect to the scaling of one-dimensional local sparseness.}
\author{Z. Bradshaw}
\author{Z. Gruji\'c}

\begin{document}

\begin{abstract}It is shown that, if the vorticity magnitude associated with a
(presumed singular) three-dimensional incompressible Navier-Stokes flow blows-up
in a manner exhibiting certain {\em time dependent local structure},
then {\em time independent} estimates on the $L^1$ norm of $|\omega|\log\sqrt{1+ |\omega|^2}$
follow.  The implication is that the volume of the region of high vorticity decays
at a rate of greater order than a rate connected to the critical scaling of
one-dimensional local sparseness and, consequently, the solution becomes sub-critical.
\end{abstract}

\maketitle
\section{Introduction}
The purpose of this article is to present blow-up scenarios under
which Leray solutions to the 3D Navier-Stokes equations (3D NSE)
behave \emph{sub-critically} with respect to the critical scaling of
\emph{one-dimensional local sparseness}. In order to contextualize
our discussion we begin with several general remarks. Speaking
informally, there is a scaling-gap between known {\em a priori}
finite quantities on one hand, e.g. $\sup_{0<t<T}||u||_2^2$ and $ \sup_{0<t<T}
||\nabla u||_{1}$, and, on the other hand, quantities with respect to which regularity can be
conditioned, e.g. $\sup_{0<t<T}||u||_3^3$ and
$\sup_{0<t<T}||\nabla u||_2^2$.  Under the natural scaling for 3D
NSE, \[u(x,t)\mapsto \frac 1 \lambda u_{\lambda}\bigg(\frac x
\lambda,\frac t {\lambda^2}\bigg),\]we see that the {\em a priori}
controlled quantities exhibit sub-critical scaling -- our examples
scale as $\lambda^1$ -- while those sufficient for regularity scale
critically as $\lambda^0$.  This mismatch is referred to as the
``scaling-gap'' and indicates the Navier-Stokes problem is
super-critical.

A theme apparent in many regularity results is the inclusion of
premises that \emph{explicitly bridge} the scaling gap. For example,
a significant result in regularity theory was non-existence of
backward-in-time self-similar blow-up (cf.
\cite{Hou-Li,NRSv,Tsai-99}). Because self-similar solutions are
scaling invariant in virtue of their construction, their study is
effectively a restriction to a class of critical solutions.  Indeed,
self-similar solutions satisfy a scaling invariant point-wise bound,
$ \operatorname{ess
\,sup} \big((|x-x_0|+\sqrt{T-t})|u(x,t)|\big)<\infty$, where the essential supremum is taken over an appropriate parabolic cylinder.  Regularity of solutions satisfying this estimate has been affirmed if the solution is additionally assumed to be
axisymmetric  (cf. \cite{ChStTsYa-II,ChStTsYa-I,KNSeSv,SeSv-08}) but the general case remains an open problem.
Regarding the axisymmetric case, the theme originally referenced is again apparent: first, the
flow is assumed to exhibit some feature which is critical and,
second, additional restrictions are identified from which regularity
follows. Similar examples, in particular those from the
$\epsilon$-regularity theory, require additional conditions on the
smallness of the presumed finite scale-invariant quantity.

In two recent complementary publications, \cite{DaGr12-1, Gr12}, a
new \emph{dynamic} approach to bridging the scaling gap is
elucidated. The physical motivation is the persistence (in the
average sense) of the axial lengths of vortex filaments.  This
picture is supported in experimental and numerical studies and by a
mathematical result which is the main result of \cite{DaGr12-1}. By
considering this length persistence and the {\em decay rate for the
volume of the region of intense vorticity}, a connection is found
between the scaling of the latter quantity and the {\em critical
scaling of one-dimensional local sparseness}.  As this argument
provides the context for our own results we make the matter more
precise.

The rigorous regularity criteria is presented in \cite{Gr12}. In
that paper, regularity of mild solutions to 3D NSE with
$L^\infty(\R^3)$ initial data is conditioned on a geometric
measure-type criteria involving one-dimensional linear sparseness of
the super-level sets of the vorticity.  The implication is that the
{\em critical scaling} for the {\em local one-dimensional
sparseness} of the region of intense vorticity is of order
$c_0||\omega(t)||_\infty^{-1/2}$.   The proof of this geometric
measure-type regularity criteria is based on an interplay among the
diffusion, the basic symmetries present in the 3D NSE and geometric
properties of the harmonic measure (resulting in strong
\emph{anisotropic} diffusion); this in turn draws on the ideas
exposed in \cite{Gr01}. On the other hand, the {\em a priori}
estimates on the $L^1$ norm of the vorticity found in \cite{Co90,
Lions} imply the volume of the region where vorticity magnitudes are
high decays according to, \[ \mbox{Vol}\bigg( \Lambda_t\bigg(\frac 1
{c_1} ||\omega(t)||_\infty\bigg)\bigg) \leq \frac {c_2}
{||\omega(t)||_\infty^{}},\] where
$\Lambda_t(y)=\{x:|\omega(x,t)|\geq y\}$. Interestingly,
if the region of intense vorticity corresponds precisely to the
space occupied by filamentary vortex structures and the length of
these filaments is non-decreasing -- that is, pinned to the
characteristic length scale of a turbulent region -- then the
anti-axial diameters of these filaments exhibit a rate of decay of
order at least $c_3||\omega(t)||_\infty^{-1/2}$, which matches the
{\em critical scaling} for {\em local one-dimensional sparseness}.
It is in this sense that the problem is rendered critical. That the
vortex filaments have persistent lengths (in the average) which are
comparable to the scale of the turbulent region is not yet
rigorously established but is supported by numerical evidence as
well as a mathematical evidence which is the prime consideration of
\cite{DaGr12-1}.

Even assuming the soundness of the above argument it is not
immediately evident that regularity follows (it would if one could
reconcile the possible difference between $c_0$ and $c_3$). The
present paper partially overcomes this by illustrating two blow-up
scenarios in which the region of intense vorticity decays at a
faster rate than the critical rate discussed in \cite{Gr12,
DaGr12-1}, thereby rendering the constants irrelevant. In
particular, we are interested in concluding,
\[  \mbox{Vol}\bigg( \Lambda_t\bigg(\frac 1 {c_1} ||\omega(t)||_\infty\bigg)\bigg)
\leq \frac {c_3}
{||\omega(t)||_\infty^{}\Phi\big(||\omega(t)||_\infty\big)},\]where,
in our first scenario, $\Phi(x)=\log(1+x)$, and, in the second,
$\Phi(x)=\log\log(e+x)$. These results are given in  Section
\ref{sec:mainresult}. We note for clarity that ours are not stated
as regularity criteria but instead scenarios under which 3D NSE
becomes {\em sub-critical} with regard to the scaling described in
\cite{DaGr12-1, Gr12}.

The above decay rates will be obtained by imposing certain
 structural requirements on the blow-up rates exhibited by
vorticity components. Let $\omega_j^+$ and $\omega_j^-$ respectively denote the positive and
negative vorticity components truncated away from zero.  We will
define  {\em amenable blow-up rates of orders 0 and 1} in detail in
Definition \ref{def:blowup} and here only illustrate a class of
functions exhibiting such blow-up profiles and describe how the
blow-up rates are connected to $\omega_j^\pm$.  A function $g$ (to
be identified with one of the $\omega_j^\pm$s) exhibits a {\em local
algebraic blow-up} (around $x_0$, at time $T$) if there exists a
constant $C>1$ so that, for $(x,t)$ in a
parabolic cylinder, $Q=B(x_0,r)\times (0,T]$, we have,
\begin{align*}\frac 1 {C} \bigg( \frac 1
{|p (x,t)|+\tau (t)}\bigg)^{\alpha (t)}\leq |g(x,t)|
\leq C \bigg(\frac {1}
{|p (x,t)|+\tau (t)}\bigg)^{\alpha (t)},
\end{align*}where, at each time $t$, $p (\cdot ,t)$ is some polynomial of
degree less than a fixed natural number $d $, $\alpha $
is positive valued and bounded away from both $0$ and $\infty$, and
$\tau$ is a positive (up to $T$) scalar function of time which
vanishes at the singular time.  The envisioned blow-up occurs at the zeros of $p(x,T)$ lying within $B(x_0,r)$.  There is a considerable amount of
freedom present in the above construction as the polynomial is
allowed to vary wildly in the time dimension. We will also
consider a scenario where some asymmetry is allowed between the
bounds assumed on vorticity components.  In particular, for
appropriate blow-up rates $D_j^\pm(x,t)$, we will require,
\begin{align*} D_j^\pm(x,t)\leq |\omega_j^\pm(x,t)| \leq CD_j^\pm(x,t)^{\beta_j^\pm(t)},
\end{align*}where $\beta_j^\pm:(0,T)\to [1,B_j]$ for some fixed value $B_j$.

A key role is played in our analysis by the {cancellations}
evident in the vortex stretching term in the context of the real
Hardy space $\mathcal{H}^1$ exploited via the Div-Curl lemma
\cite{CoLiMeSe} and the $\mathcal{H}^1-BMO$ duality
\cite{FeSt-72,Fe-72}. In the standard way this gives uniform-in-time
control of the vortex stretching term. The structural blow-up
assumptions are provided to ensure uniform-in-time control of the
$BMO$ norm of a multiplier (multiplied against the vortex stretching
term). In the algebraic case, this is enabled by the logarithm's
depletive effect on the unboundedness of the mean oscillations of
polynomial functions.  The effectiveness is witnessed by the
remarkable fact (cf. a proof by Stein in \cite{BeijingLectures})
that there exists a constant $C(d,n)$ so that, for any polynomial on
$\R^n$ of degree less than or equal to $d$, \[\big|\big|\log
|P|\big|\big|_{BMO}\leq C(d,n).\] In particular, the constant is
\emph{independent of the coefficients}. This allows us to introduce
{\em time-dependent algebraic comparability conditions} on the
spatial profiles prior to a possibly singular time and do so in a
manner that preserves {\em time-independent estimates} on the $BMO$
norms of the logarithms of these profiles.

We proceed in Section \ref{sec:preliminaries} to review needed
results from harmonic analysis and then define the classes of
blow-up scenarios which will be amenable to our PDE argument.  We
also include results which connect the structural blow-up
assumptions to an energy inequality-type argument given in Section
\ref{sec:mainresult}.  The statements and proofs of the main results
are contained in Section \ref{sec:mainresult}.

%********************************************************************
%***    SECTION 2    ****
%********************************************************************
\section{Preliminaries and Amenable Blow-Up Rates} \label{sec:preliminaries}

%********************************************************************
%***    HARMONIC ANALYSIS    ****
%********************************************************************
Here we review needed results from harmonic analysis and present a
lemma which will connect these ideas to the PDE context of Section
\ref{sec:mainresult}.  Following \cite{St93}, the maximal function of a
distribution $f$ is defined for all $x\in \R^n$ as,
\[M_hf(x)=\sup_{t>0}|f*h_t(x)|,
\] where $h$ is a fixed test function supported on the unit ball so that $\int h ~dx=1$ and $h_t$ denotes $t^{-n}h(x/t)$.

\begin{definition}
The distribution $f$ is in the Hardy space $\mathcal H^1$ if $||f||_{\mathcal H^1}:=||M_hf||_1<\infty$.
\end{definition}

In \cite{CoLiMeSe}, Coifman, Lions, Meyer, and Semmes reformulated and refined some
key features of the `sequential' theory of compensated compactness within the framework
of Hardy spaces, the key idea being that certain \emph{nonlinear quantities} exhibit \emph{cancelations} yielding the improved regularity. One such result is the Div-Curl lemma.

\begin{lemma}(Coifmann, Lions, Meyer, Semmes -- \cite{CoLiMeSe}) Suppose
$E,B\in L^2 (\R^3)^3$ with $\nabla \cdot E=0$ and $\nabla \times B
=0$ (in the sense of distributions).  Then, $E\cdot B\in \mathcal
H^1$, and there exists a universal constant $C$ such that
\[||E\cdot B ||_{\mathcal H^1}\leq C ||E||_{L^2}||B||_{L^2}.\]
\end{lemma}
Because weak solutions to the Navier-Stokes equations are divergence
free and because the curl of a gradient is always equal to the
trivial distribution, it follows that the advective term,
$(u\cdot\nabla )u$, in the velocity-pressure formulation exhibits
div-curl structure, as does the vortex stretching term,
$(\omega\cdot\nabla)u$, in the vorticity-velocity formulation.  Consequently,
using Hardy spaces, refined regularity results can be established
for weak solutions of 3D NSE (see Chapter 3.2 of \cite{Lions} for a
collection of such results).

The well known result of Fefferman (cf. \cite{Fe-72,FeSt-72}, as
well as Stein's monograph \cite{St93}) establishes that the dual
space of $\mathcal H^1$ is precisely the space of functions of
bounded mean oscillation, $BMO$. By $f_B$ we denote the average of
the locally integrable function $f$ over the ball $B$ of volume
$|B|$; i.e., $f_B=|B|^{-1}\int_Bf~dx$.  The mean oscillation of $f$
over $B$ is  $|B|^{-1}\int |f-f_B|~dx$ and is the quantity typically
used to characterize $BMO$.
\begin{definition}The locally integrable function $f$ is in $BMO$ if  $||f||_{BMO}<\infty$ where
 \[||f||_{BMO}:=\sup_{x\in \R^3;0<r}\frac 1 {|B(x,r)|}\int_{B(x,r)} \big|f(y)-f_{B(x,r)} \big|~dy.\]
\end{definition}

Several results about the space $BMO$ will expedite future work. An
elementary fact is that, if $g\in BMO$, then the
following implication holds,
\begin{align}\label{implication:bmointerp}g(x)+m\leq f(x)\leq g(x)+M
\,\mbox{for}\, \operatorname{a.e.} \, x\in\R^3  \quad \Longrightarrow  \quad ||f||_{BMO}\leq
||g||_{BMO}+M-m.\end{align}  The above is easy to see by directly
comparing $|f-f_B|$ to $|g+M-(g+m)_B|$ and concluding that, for all
balls $B$, we have,
\[\frac 1 {|B|} \int_B \big| f-f_B\big|~dx\leq \frac 1 {|B|}\int_B \big| g-g_B\big|~dx+M-m.\]
We include a technical lemma regarding certain point-wise
multipliers on $BMO$ which appears in a paper by Iwaniec and Verde
(cf. Lemma 2.1 of \cite{IwVe-99}; see also Lemma 5.10 of
\cite{CDS99}). This result will allow us to use the standard
$\mathcal H^1$-$BMO$ duality in an essentially local context instead
of working in the local versions of these spaces. It is worth
mentioning that our results can be achieved in the local spaces
using a local non-homogeneous version of the Div-Curl lemma (cf. \cite{CoLiMeSe})
paired with various local $h^1-bmo$ dualities (cf. \cite{CDS99}) thereby obtaining estimates on the localized vortex stretching term (this approach is similar to the estimates in \cite{GrGu2}).

\begin{lemma}\label{lemma:bmomultiplier}(Iwaniec and Verde -- \cite{IwVe-99}) Suppose
$h\in BMO (\R^n)$ and $\phi\in C_0^1(\R^n)$ is supported on the ball $B$.  Then $\phi h\in
BMO(\R^n)$ and we have,\[||\phi h ||_{BMO}\leq C(n,R) ||\nabla
\phi||_\infty\bigg( (||h||_{BMO}+\bigg|\frac 1 {|B|}\int_B h ~dx
\bigg| \bigg).\]
\end{lemma}
The space $BMO$ enjoys an intimate connection with the logarithm.
This is illustrated in the following lemma due to Stein.

\begin{lemma}\label{lemma:bmoandpolynomials}(Stein -- \cite{BeijingLectures})
Let $P$ be any polynomial in $\R^n$ of degree less than or equal to
$d$.  Then there exists a constant $C=C(d,n)$ so that $\log|P|\in
BMO$ and
\[||\log |P|\,||_{BMO}\leq C(d,n).\]
\end{lemma}
The fact that the constant appearing above is \emph{independent of the
coefficients} evidences the remarkable extent to which the logarithm
depletes mean oscillations and motivates the assumptions on which the
conclusions of Section \ref{sec:mainresult} are conditioned. This
will be done by assuming that solutions blow-up in a fashion
possessing certain structure.  The remainder of this section is
dedicated to describing amenable blow-up structures and the
statement and proof of a lemma which will connect this structure to
the PDE context of \ref{sec:mainresult}.  We denote by $\log^m(x)$
an iterated composition of the logarithm with itself $m-1$ times,
that is, $\log^m(x)=\log\circ\cdots\circ \log(\sqrt{e_m+x^2})$ where
$m$ applications the logarithm are carried out and $e_m$ is a power
of $e$ defined so that $\log^m(0)=0$. 
We will use the following notation when defining functions of $(x,t)$ possessing singularities at time $T$ in the spatial set $S\subset \R^3$, \[Q_T(x_0,R,S)= B(x_0,R)\times(0,T) \cup((B(x_0,R)\setminus S)\times \{T\}~ \mbox{and}~\Omega_T(S)=\R^3\times (0,T) \cup (\R^3\setminus S)\times \{T\},\]where $R>0$ and $x_0\in \R^3$.  In contexts where $x_0$ and $R$  are fixed we will make the abbreviation $Q_T(S)=Q_T(x_0,R,S)$.

%***********************************************************************
%
%  Blow up rates and discussion
%
%***********************************************************************

\begin{definition}\label{blowup}Fix $x_0\in \R^3$, $R,T>0$, $m \in \N_0$ and a set of measure zero, $S$, which is contained in a compact subset of $B(x_0,R)$.
\begin{itemize}\item[a.] An {\em amenable blow-up rate of order $m$ on $Q_T(S)$} is
a function, $D(x,t):\Omega_T(S)  \to [0,\infty)$, which additionally satisfies,
\begin{itemize}
\item[i.] there exists $M_0>0$ so that $ \sup_{0< t\leq T}||\log^{m+1} D(x,t)||_{BMO}<M_0$,
\item[ii.] there exist $M_1, M_2>0$ so that $1/M_1\leq D(x,t)$ on $Q_T(S)$ and
$D(x,t)\leq M_2$ on $B(x_0,R)^c\times (0,T]$.
\end{itemize}
\item[b.] \label{def:blowup} A function $f(x,t):Q_T(S)\to \R$ {\em exhibits an
amenable blow-up rate of order $0$ on $Q_T(S)$} if there exists an
amenable blow-up rate of order $0$, $D(x,t)$, on $Q_T(S)$, so that, for
some $C_*>1$,
\[D(x,t)\leq |f(x,t)|\leq C_*D(x,t)\quad \mbox{for all}~(x,t)\in Q_T(S).\]
\item[c.] \label{def:blowup2} A function $f(x,t):Q_T(S)\to \R$ {\em exhibits an
amenable blow-up rate of order $1$ on $Q_T(S)$} if there exists an
amenable blow-up rate of order $1$, $D(x,t)$, on $Q_T(S)$, so that, for
some $C_*>0$, $A\geq 1$, and a scalar function $\alpha:[0,T]\to
[1,A]$,
\[D(x,t)\leq |f(x,t)|\leq C_*D(x,t)^{\alpha(t)}\quad \mbox{for all}~(x,t)\in Q_T(S).\]
\end{itemize}
\end{definition}

The condition (\ref{blowup}.a.i) will prove crucial in establishing our
estimates in Section \ref{sec:mainresult}.  At face value, however,
it is not clearly motivated. To address this we specify an
expansive class of functions which are simultaneously amenable
blow-up rates of order $0$ and are reasonable blow-up scenarios
given what is known about the structure of possible singularities in
weak solutions of 3D NSE.  In the following, $\R[x_1,x_2,x_3]$ denotes the ring of tri-variate polynomials with coefficients in $\R$.

\begin{definition}\label{def:algblowup}Use the notations of Definition \ref{blowup} and fix $d\in \N$. The function
$D(x,t):\Omega_{T}(S)\to [0,\infty)$ is an {\em algebraic blow-up
rate of degree $d$ on $Q_T(S)$} if there exist functions $\tau:(0,T]\to
[0,\infty)$,  $\alpha:(0,T]\to (0,\infty)$, and
$\rho(x,\cdot):(0,T]\to \R[x_1,x_2,x_3]$ so
that,\[D(x,t)=\bigg(\frac 1
{|\rho(x,t)|+\tau(t)}\bigg)^{\alpha(t)}~\mbox{on}~Q_T(S),\] and we
additionally have,
\begin{itemize}
\item[i.] for all $t\in (0,T]$, $\rho(x,t)$ is a polynomial of degree less than or equal to $d$ and the zeros of $\rho(x,T)$ are contained in $S$,
\item[ii.] there exists $A\geq 1$ so that $\alpha$ takes values in $[A^{-1},A]$,
\item[iii.] $\tau(t)>0$ for $t\in (0,T)$ and vanishes as $t$ approaches $T$,
\item[iv.] the condition (\ref{def:blowup}.a.ii.) is satisfied.
%\item the zeros of $\rho(x,T)$ form a compact subset of $Q_0$.
\end{itemize}
\end{definition}

\begin{remark}\label{remark:algblowup} We have presented these algebraic blow-up
rates because they constitute a concrete class of amenable blow-up
rates of order $0$. To verify this, we check that condition
(\ref{def:blowup}.a.i.) is satisfied, which is clear if we first
expand $\log D(x,t)$ as,
\[\log(D(x,t))=-\alpha(t) \log\big(|\rho(x,t)|+\tau(t)\big),\]and, additionally observe that,
\[\max\{\log|\rho(x,t)|,\log\tau(t)\}\leq \log(|\rho(x,t)|+\tau(t))
\leq \log(2)+\max\{\log|\rho(x,t)|,\log\tau(t)\}.\] Recalling the
fact that, if $f,g\in BMO$, then $||\max \{ f,g \}||_{BMO}\leq
2(||f||_{BMO}+||g||_{BMO})$ ($BMO$ is a lattice),  we are able to conclude by applying the
implication \eqref{implication:bmointerp} in conjunction with Lemma
\ref{lemma:bmoandpolynomials}.
\end{remark}

Our main lemma contextualizes amenable blow-up rates to their
application in Section \ref{sec:mainresult}.  As presently defined,
the blow-up rates require that $f(x,t)$ (as given in Definition
\ref{def:blowup}) is bounded away from zero.  To accommodate
functions possibly  {\em not} bounded away from zero we employ an
auxiliary function, $F=\sqrt{e_m+f^2}$, where $e_m$ is determined by
the order of the amenable blow-up rate in question. Based on the
sign of $f$ we define,
\begin{align}\label{decomp:F}F_+ =\begin{cases}
F, & \text{if }f\text{ is non-negative,} \\
1, & \text{otherwise, }
\end{cases} \qquad F_-=\begin{cases}
1, & \text{if }f\text{ is non-negative,} \\
F, & \text{otherwise. }\end{cases}
\end{align}The factorization $F^{\sign f}=F_+(F_-)^{-1}$ will allow us to
independently impose blow-up assumptions on a function's positive and negative components.

%***************************************************************
%
%   MAIN LEMMA
%
%***************************************************************

\begin{lemma}\label{lemma:uniformbmo}Fix $R_0>0$.  
Let $S'\subset \R^3$ be a set of Lebesque measure zero for which $S'\cap B(0,R_0)\subset B(0,R_0')$ where $0<R_0'<R_0$.  Suppose the function $f:\Omega_T(S') \to \R$ satisfies
$f\in L^\infty ((0,T];L^1(B(0,R_0)))$. Fix $\phi \in C_0^1(\R^3)$ so that $\phi$ is supported on $B(0,R)$ and let $S$ be a set of measure zero satisfying $S'\cap B(0,R_0)\subset S \subset B(0,R_0)$.
\begin{itemize}\item[a.] If $f|_{Q_T(S)}$ exhibits an amenable blow-up of order $0$ in $Q_T(S)$, then,
\begin{align*}\sup_{0<t\leq T}\big|\big|\phi(x)\log |f(x,t)| \big|\big|_{BMO}<\infty.\end{align*}
\item[b.]If $F_+|_{Q_T(S)}$ and $F_-|_{Q_T(S)}$ exhibit amenable blow-up rates of order $0$ in $Q_T(S)$,
then, \[\sup_{0<t\leq T}\bigg|\bigg|\phi(x) \frac {f(x,t)}{F(x,t)}\log
F(x,t) \bigg|\bigg|_{BMO}<\infty.\]
\item[c.]If $F_+|_{Q_T(S)}$ and $F_-|_{Q_T(S)}$ exhibit amenable blow-up rates of order $1$ in $Q_T(S)$,
then, \[\sup_{0<t\leq T}\bigg|\bigg| \phi(x)\frac
{f(x,t)}{F(x,t)}\log\log F(x,t) \bigg|\bigg|_{BMO}<\infty.\]
\end{itemize}
\end{lemma}

\begin{proof}
{\em (Part a.)} Let $D(x,t)$ be the amenable blow-up rate exhibited
by $f|_{Q_T(S)}$.  We will freely reference the constants associated with $D$
in the statement of Definition \ref{def:blowup} and subsequently make the abbreviations $Q_T(S)=Q_T(0,R_0,S)$ and $B=B(0,R_0)$. Define an
extension, $\widetilde F$, of $f|_{Q_T(S)}$ to $\Omega_T(S)$ by $\widetilde
F=|f|$ on $Q_T(S)$ and  $\widetilde F =1$ on $B ^c\times (0,T]$. Set $\widetilde
D(x,t)=\max \{ D(x,t),M_1^{-1} \}$ and $\widetilde
C_*=\max\{C_*,M_2\}$.  These definitions ensure that, for all
$(x,t)\in\Omega_T(S)$,\[\frac 1 {\widetilde C_{*}}\widetilde
D(x,t)\leq\widetilde F(x,t)\leq \widetilde C_{*}\widetilde
D(x,t),\]and, taking logarithms,
\[ \log\big( \widetilde D(x,t)\big)-\log \widetilde C_{*}
\leq \log\big(\widetilde F(x,t)\big)\leq  \log\big(\widetilde
D(x,t)\big)+\log \widetilde C_{*}.\] Recalling the implication
\eqref{implication:bmointerp}, we see that,\[||\log \widetilde
F(x,t)||_{BMO}\leq C(M_0,\tilde C_*).\]  Because $|f|=\widetilde F$ on the
support of $\phi$ we have,
\begin{align*}|| \phi\log |f|\, ||_{BMO}&= ||\phi \log \widetilde F||_{BMO}
\\&\leq C \, ||\nabla \phi||_\infty \bigg( ||\log \widetilde F||_{BMO}+ \bigg|
\frac 1 {|B|}\int_B \log \widetilde F ~dx \bigg|\bigg),
\end{align*}
where we have applied Lemma \ref{lemma:bmomultiplier}. The average
appearing above is uniformly bounded in time; this follows from the
facts that $M_1^{-1}\leq |f|$ and $f\in L^\infty
((0,T];L^1(B(0,R_0)))$.
%****************************************************************************
%****************************************************************************
\\\smallskip

{\em (Part b.)}  We begin by remarking that, for all $y\in \R$,
\[(\sign y) \log(\sqrt{1+y^2})-1\leq \frac y {\sqrt{1+y^2}}
\log(\sqrt{1+y^2})\leq (\sign y) \log \sqrt{1+y^2})+1.\] Taking
$y=|f(x,t)|$, in light of implication \eqref{implication:bmointerp}
it is sufficient to work with the function $(\sign f)\log F$, where
$F$, and also $F_+$ and $F_-$, are defined as in the comments
preceding the statement of the lemma. Those
definitions were motivated by the fact that $F^{\sign
f}=F_+(F_-)^{-1}$ and, so,  \[\log \big(F^{\sign f}
\big)=\log(F_+)-\log(F_-).\]

Let $G$ be the extension by $1$ of $(fF^{-1}\log(F))|_{Q_T(S)}$ from
$Q_T(S)$ to $\Omega_T(S)$ and adopt the notation from the proof of part a.~of
this lemma for definitions of functions analogous to those mentioned
above where we designate by $D_\pm$ the blow-up rates exhibited by
$F_\pm$.   Then, on $\Omega_T(S)$, we have,
\[\log(\widetilde F_+)-\log(\widetilde F_-)-1\leq G\leq
\log(\widetilde F_+)-\log(\widetilde F_-)+1,\]and, therefore,
\begin{align*}||G||_{BMO}&\leq  ||\log(\widetilde F_+)||_{BMO}+||\log(\widetilde F_-)||_{BMO} + 2.
\end{align*}
Appealing to part a. of the lemma, the dominating quantity is itself bounded uniformly in time.
Extending this to the desired estimate proceeds directly,
\begin{align*}\bigg|\bigg|\phi \frac f {F} \log(F)\bigg|\bigg|_{BMO}&\leq C \, ||\nabla \phi||_\infty
\bigg(||G||_{BMO}+ \bigg|\frac 1 {B}\int_B \log \widetilde G ~dx \bigg| \bigg),\end{align*}
which, again, is finite by assumptions on $f$.
%****************************************************************************
%****************************************************************************
\\\smallskip

{\em (Part c.)} We begin similarly, noting,
\[(\sign f)\log \log F-1\leq \frac {f} {F}\log \log F\leq(\sign f)\log \log F +1.\]This
 leads us to consider, \[\log \bigg( \big(\log F\big)^{\sign f}\bigg)=\log
 \bigg( \big(\log F\big)_+\bigg)+\log\bigg( \big(\log F\big)_-\bigg),\]
where, \[ \big(\log F\big)_+=\max\bigg\{\big(\log F\big)^{\sign
f},1\bigg\} \quad\mbox{and}\quad  \big(\log
F\big)_-=\min\bigg\{\big(\log F\big)^{\sign f},1\bigg\}.\] It is easy to check that $(\log
F)_\pm=\log(F_\pm)$, and, therefore,\[\frac {f} {F}\log \log F=
\log\log(F_+)-\log\log(F_-).\] A further observation is that, taking
$D_{\pm}(x,t)$ to be the amenable blow-up rates exhibited by
$F_{\pm}$, and setting $k_\pm^*=\log(A_\pm M_{1,\pm}+\log
C_{*,\pm})$ (these are constants associated with $D_\pm$ as in
Definition \ref{def:blowup}.a.), we have for $(x,t)\in Q_T$,
\[\log\log (C_*D_\pm(x,t)^{\alpha_\pm(t)})\leq  \log\log(D_\pm(x,t)
)+k_\pm^*,\]and, recalling the comparability condition from
Definition \ref{def:blowup}.c., this implies that,
\[\log\log(D_\pm(x,t))\leq F_\pm \leq \log\log(D_\pm(x,t) )+k_\pm^*.\]
The conclusion now follows in the same fashion laid out in part b.; we omit the details.
\end{proof}

%*******************************************************************
%***
%***    MAIN RESULTS
%***
%********************************************************************

\section{Main Results}\label{sec:mainresult}

Weak solutions to the 3D incompressible Navier-Stokes equations are
functions which satisfy (distributionally) the following system of
PDEs,
\begin{equation}
\label{eq:NSE}
\left\{  \begin{array}{c}
                       \partial_t u +(u\cdot\nabla)u = -\nabla p + \nu \Delta u, \\
                        \nabla\cdot u=0; \quad u(x,0)=u_0\in H, \\
           \end{array}  \right.\end{equation}
where $H$ is the $L^2$ closure of the divergence free test functions and the
initial datum is understood in the sense of weak continuity
(cf. \cite{CF} for details). The evolution of the
vorticity, $\omega=\nabla\times u$, is of special interest to us and
satisfies,
\begin{equation}
\label{eq:vorticity}
\left\{  \begin{array}{c}
                       \partial_t \omega + (u\cdot \nabla) \omega = (\omega\cdot \nabla)
                       + \nu \Delta \omega, \\
                        \nabla\cdot \omega=0. \\
           \end{array}  \right.
\end{equation}

For simplicity, we consider a weak solution on $\R^3\times
(0,\infty)$, evolving from the initial data $u_0$.  We also require that the initial vorticity, $\omega_0 = \nabla \times u_0$, is in $
L^1 \cap L^2$.

The standard regularity results for weak solutions (cf. \cite{CF}) consist of the {\em a
priori} bounds,\[\sup_{0<t<T'}||u||_{L^2(\R^3)}^2 <\infty
\quad\mbox{and} \quad \int_0^{T'} ||\nabla
u||_{L^2(\R^3)}^2~dt<\infty,\] for any $T' > 0$. In addition, since
$\omega_0 \in L^1$, a result from \cite{Co90} ensures that,
\[
\sup_{0\leq t\leq T'}||\omega||_{L^1(\R^3)} <\infty.
\]
Note that since $\omega_0 \in L^2$, our weak solution
locally-in-time coincides with the smooth solution; let $T$ be the
first (possible) singular time.

Fix a `macro-scale,' $R_0>0$, with the property
that the intersection of $B(0,R_0) \times \{T\}$ with the singular set at time $T$ is nonempty. Fix $0<\epsilon<R_0$. Our estimates are intended for integrals over the spatial set $B(0,R_0-\epsilon)$ and localization is
achieved via multiplying equation \eqref{eq:vorticity} by a smooth
cut-off function $\psi$ satisfying,
\[\supp \psi \subset B(0,R_0
),\quad \psi=1~\mbox{on}~B(0,R_0-\epsilon),\quad\frac {|\nabla
\psi|} {\psi^\rho}\leq \frac c \epsilon~\mbox{for some}~\rho\in
(0,1)\quad \mbox{and}\quad 0\leq \psi\leq 1.\] Instead of studying
the evolution of $|\omega|$ and $|\omega_k|$ directly we introduce
an auxiliary function.  This approach is an adaptation of that taken
by Constantin in \cite{Co90}. Define $q(y)=\sqrt{1+y^2}:\R\to \R$
and let $w_k=q( \omega_k):\R^3\to \R$. From these definitions it is
immediate that,
\begin{align}\label{list:factsaboutq}|\omega_k|\leq w_k,\quad -1\leq q'(\omega_k)=
\frac {\omega_k} {w_k}\leq 1, \quad\mbox{and} \quad0< q''(
\omega_k)=\frac {1} {w_k^3}\leq 1.
\end{align}
For convenience we recall the following elementary facts (for $y\geq 1$),
\begin{align}\label{list:factsaboutlog}0\leq -\Log''(y)\leq \Log'(y)\leq 1,\quad
y\Log'y= 1,\quad \mbox{and}\quad 0= \Log'(y)+y\Log''(y).
\end{align}
To allow $\omega_k$ to have varying sign we define the functions
$w_{k,+}$ and $w_{k,-}$ in the same manner as \eqref{decomp:F}. For convenience we recall the notation 
$\Lambda_t(y)=\{x:|\omega(x,t)|\geq y\}$. Also, throughout the remainder of this paper, $c_1$ denotes a fixed constant taken to be greater than one. The role of $c_1$ lies in specifying the threshold for the super-level sets $\Lambda_t(y)$ as fractions of the supremum norm of the modulus of the vorticity at the time $t$, i.e. $y=c_1^{-1}||\omega(\cdot,t)||_\infty$.

We include two theorems, one each for the case of solutions
possessing vorticity components which exhibit amenable blow-up rates
of order $0$ and of order $1$.  The proof of the
order $0$ case extends easily to the case of blow-ups of order $1$
and the proof of the second result is accordingly terse. 

%******************************************************************
% MAIN RESULT
%******************************************************************
\begin{theorem}\label{thrm:main}Let $u$ be a Leray solution to
\eqref{eq:NSE} on $\mathbb{R}^3 \times (0,\infty)$, and suppose additionally that $\omega_0 = \nabla \times u_0 \in L^1 \cap L^2$.
Denote by $T$ the first singular time and by $S$ the singular set of $\omega$ at time $T$.  Fix positive values $R_0$ and $\epsilon$ so that $\emptyset\neq S\cap B(0,R_0)\subset B(0,R_0-\epsilon)$ -- i.e. there are singular points in $B(0,R_0)$ but all such points are in $B(0,R_0-\epsilon)$ -- and let $Q_T(S)=Q_T(0,R_0,S)$.
\begin{itemize}\item[(i)] If $w_{k,+}|_{Q_T(S)}$ and $w_{k,-}|_{Q_T(S)}$ each exhibit amenable blow-up rates of
order $0$ in $Q_T(S)$, then there exists a positive value $M_k$ for which, \begin{align*}\sup_{t\in [0,T]}
\int_{B(0,R_0-\epsilon)}|\omega_k(x,t )| \log
\big(\sqrt{1+|\omega_k(x,t)|^2}\big)~dx<M_k.
\end{align*}
\item[(ii)] Let $c_1$ ba a fixed constant which is greater than $1$. If, for all $k$, the premises of part (i) are satisfied, then there exists a positive constant $M_0$ so that, for $t\in [0,T]$,
 \begin{align*} \mbox{Vol}\bigg( \Lambda_t\bigg(\frac 1 {c_1}
 ||\omega(t)||_{L^\infty(B(0,R_0-\epsilon))}\bigg)\bigg) \leq \frac {M_0}
 {||\omega(t)||_{L^\infty(B(0,R_0-\epsilon))}\log \big(1+||\omega(t)||_{L^\infty(B(0,R_0-\epsilon))}\big)}.
 \end{align*}
\end{itemize}
\end{theorem}

\begin{remark}
As will be seen, our energy inequality-type method depends heavily
on classical techniques and we thus need to consider smooth
solutions.  It is possible to extend the estimates to {\em some}
weak solutions by considering a sequence of smooth approximations --
a good choice to use here is the method of retarded mollifiers given
for suitable weak solutions  in \cite{CKN-82} as it allows one to
recover from a velocity-level approximation scheme information about
the vorticity, see \cite{Co90} for more discussion -- but, as our
estimates hinge on assumptions beyond the initial data, $u_0$, we
would have to ensure these are met by a convergent sub-sequence of
approximate solutions.
\end{remark}

\begin{proof}
%******************************************************************
%ENERGY INEQUALITY-TYPE ESTIMATE
%******************************************************************
In virtue of our smoothness assumption, the evolution of $\psi
w_k\Log (w_k)$ can be established from the evolution of $\omega_k$
by first writing,
\begin{align*}&\partial_t\big(\psi w_k\Log w_k \big)
\\&=  \psi q'(\omega_k)\big( \Log w_k + w_k\Log'w_k\big)\big( \partial_t \omega_k \big)
\\&= \psi q'( \omega_k)\big( \Log w_k + w_k\Log'w_k\big)\big(\nu\Delta \omega_k -(u\cdot \nabla )
\omega_k+(\omega\cdot \nabla) u_k\big),
\end{align*}
and then deriving (noting tacit summation over terms involving
indices other than $k$),
\begin{align*}&\partial_t(\psi w_k \Log w_k)
-\nu \psi  \Delta w_k \big(\Log w_k + w_k\Log'w_k  \big) +\nu\psi
q''(\omega_k) (\partial_i \omega_k)^2\big(\Log w_k + w_k\Log'w_k
\big)
\\\notag &= \psi(\omega\cdot \nabla) u_k q'(\omega_k)\big(\Log w_k + w_k\Log'w_k  \big)
- \psi (u\cdot \nabla) w_k\big(\Log w_k + w_k\Log'w_k\big).
\end{align*}
By integrating in space and time and dropping the positive quantity
involving $q''$ from the left hand side, we obtain the following
energy inequality-type estimate,
\begin{align}\label{ineq:rawEnergyEstimate}&\int \psi(t) w_k(t)\Log w_k(t) ~dx -\nu
\int_0^t\int \psi  \Delta w_k \big(\Log w_k + w_k\Log'w_k
\big)~dx~ds
\\\notag &\leq \int_0^t\int \psi(\omega\cdot \nabla) u_k q'( \omega_k)
\big(\Log w_k + w_k\Log'w_k  \big)~dx~ds
\\\notag &\quad -\int_0^t\int \psi (u\cdot \nabla) w_k\big(\Log w_k + w_k\Log'w_k\big)~dx~ds
\\\notag &\quad +\int \psi\, w_{k,0}\Log w_{k,0}~dx.
\end{align}
The properties in \eqref{list:factsaboutlog} enable several key
cancellations. For the dissipative terms, integration by parts
reveals that,
\begin{align*} \int_0^t\int (\partial_j^2 w_k)\psi \Log w_k ~dx~ds&= \int_0^t\int
\bigg( w_k(\partial_j^2 \psi)\Log w_k +w_k (\partial_j \psi)(\partial_jw_k)\Log' w_k  \bigg)~dx~ds
\\&\quad - \int_0^t\int (\partial_j w_k)^2 \psi \Log' w_k  ~dx~ds,\end{align*}and,
\begin{align*}\int_0^t\int (\partial_j^2 w_k)\psi w_k\Log' w_k ~dx~ds&=-\int_0^t\int w_k
(\partial_j \psi)(\partial_jw_k)\Log' w_k ~dx~ds
\\&\quad -\int_0^t\int(\partial_j w_k)^2\psi \bigg(\Log' w_k  + w_k \Log'' w_k \bigg)~dx~ds.
\end{align*}A cancellation occurs upon adding the above equations leaving us with,
\begin{align*}&-\nu \int_0^t\int (\partial_j^2 w_k)\psi\bigg( \Log w_k + w_k\Log' w_k \bigg)~dx~ds
\\&=\nu \int_0^t\int (\partial_j w_k)^2 \psi\bigg(  2\Log' w_k +w_k\Log'' w_k\bigg)~dx~ds
 -\nu \int_0^t\int w_k(\partial_j^2 \psi)\Log w_k~dx~ds.\end{align*}Again noting
 \eqref{list:factsaboutlog}, the integrand of the first term is positive and, therefore,
 can be dropped from the left hand side of estimate \eqref{ineq:rawEnergyEstimate}. The
 second term can be dominated by an {\em a priori} finite quantity arising from the
 standard energy inequality for weak solutions in conjunction with the fact that our spatial
 integral is over a set of finite measure.  More precisely,
\begin{align}\label{ineq:trivial}\int_0^t\int w_k(\partial_j^2 \psi)\Log w_k~dx~ds&\leq
 C \int_0^T \int_{B(0,R_0)}w_k^2~dx~ds
\\\notag&= C \int_0^T \int_{B(0,R_0)}|\omega_k|^2~dx~ds+T|B(0,R_0)|.
\end{align}
The integrals arising from the transport term in
\eqref{eq:vorticity} also enjoy substantial cancellations,
\begin{align*}&\int_0^t\int \psi~u\cdot \nabla w_k  \bigg(\Log w_k+w_k \Log'w_k  \bigg)~dx~ds
\\&=\int_0^t\int \psi~ u_j(\partial_j w_k )w_k \Log' w_k ~dx~ds
- \int_0^t\int \psi~ u_j w_k (\partial_jw_k)\Log' w_k ~dx~ds
\\&\quad-\int_0^t\int u_jw_k\Log w_k \partial_j \psi~dx~ds\\
&\quad = -\int_0^t\int u_jw_k\Log w_k \partial_j \psi~dx~ds.
\end{align*}
Noting that $\log(y)\leq 4 \, y^{1/4}$,
\begin{align*}&
\bigg| \int_0^t\int u_jw_k\Log w_k |\partial_j \psi|~dx~ds \bigg|
\\&\leq C  \int_0^T \int \big(|\partial_j\psi|^{1/4}|u_j|\big) \big(|\partial_j
\psi|^{1/2}|w_k|\big)  |w_k \partial_j\psi|^{1/4} ~dx~ds
\\&\leq C\int_0^T ||\psi^{\rho/4}u||_4
||w_k ||_{L^1(B(0,R_0))}^{1/4}||w_k||_{L^2(B(0,R_0))}
\\&\leq C\sup_{0<t\leq T}||w_k ||_{L^1(B(0,R_0))}^{1/4}\int_0^T ||\nabla(\psi^{\rho/4}u)||_2
||w_k||_{L^2(B(0,R_0))},
\end{align*}where we have used H\"older's inequality and the Sobolev inequality.
From here and in light of \eqref{ineq:trivial}, a commutator estimate on the gradient of the localized
velocity allows the extension of the above estimate to
one in terms of {\em a priori} finite quantities.

The integral on the last line of the right hand side of
\eqref{ineq:rawEnergyEstimate} is finite by our assumptions on the
initial data.

At this point, based on \eqref{ineq:rawEnergyEstimate}, we have established,
\[\int (\psi w_k\Log w_k)(t)~dx  \leq  \bigg|\int_0^t\int
\psi ~\omega\cdot \nabla u_k \, \,q'( \omega_k ) \bigg( \Log w_k
+w_k\Log'w_k\bigg)~dx~ds\bigg|+R,\]where $R$ is comprised of those
{\em a priori} bounded quantities accumulated in the preceding
estimates.  Noting that,\[\psi q'(\omega_k)w_k\Log' w_k\leq 1,\] an
{\em a priori} bound follows for that part of the as of yet
unbounded quantity leaving us with,\[\int (\psi w_k\Log w_k)(t)~dx
\leq  \bigg|\int_0^t\int   \psi ~\omega\cdot \nabla u_k ~ \frac
{\omega_k} {w_k}\Log w_k  ~dx~ds\bigg|+R.\]

%******************************************************************
%TREATMENT FOR VORTEX STRETCHING TERM
%******************************************************************

The $\mathcal H^1$-$BMO$ duality and the Div-Curl lemma justify the
following chain of inequalities,
\begin{align*} &\bigg|\int_0^t\int   \bigg(\omega\cdot \nabla u_k \bigg)
\bigg( \psi\frac {\omega_k} {w_k}\Log w_k \bigg)  ~dx ~ds\bigg|
\\& \leq \int_0^T \big| \big|\omega\cdot \nabla u_k \big|\big|_{\mathcal H^1} \bigg|\bigg|\psi
\frac {\omega_k} {w_k}\Log w_k\bigg|\bigg|_{BMO} ~ds
\\&  \leq \bigg(\sup_{0<t\leq T} \bigg|\bigg|\psi\frac {\omega_k}
{w_k}\Log w_k\bigg|\bigg|_{BMO}\bigg) \int_0^T\big| \big|\omega\cdot
\nabla u_k\big|\big|_{\mathcal H^1}~ds
\\&\leq  \bigg(\sup_{0<t\leq T} \bigg|\bigg|\psi\frac {\omega_k}
{w_k}\Log w_k\bigg|\bigg|_{BMO}\bigg) \bigg(\int_0^T
||\omega||_2^2~dt\bigg)^{1/2} \bigg(\int_0^T||\nabla
u||_2^2~dt\bigg)^{1/2}. \end{align*} By Lemma \ref{lemma:uniformbmo}
and the standard regularity of Leray weak solutions all of the above
are finite and we have thus established that, for all $0<t\leq T$,
$||\psi w_k\Log w_k(t)||_{L^1(\R^3)}$ is majorized by
time-independent {\em a priori} bounded quantities, the sum of which
we label $M_k$. This completes our proof of part (i) of the theorem.

%******************************************************************
%PROOF OF PART (ii)
%******************************************************************
Part (ii) of the theorem is proven in two
steps. For the first, let, \[\lambda(t)=\frac 1 {c_1} ||\omega(t)||_{L^\infty(B(0,R_0-\epsilon))},\]and observe that, for any $x\in B(0,R_0-\epsilon)$ where $|\omega(x,t)|\geq \lambda(t)$, direct computation affirms that, \[1\leq \frac {c_1|\omega(x,t)|\big[\log(c_1)+\log\big(1+|\omega(x,t)|\big) \big]}{||\omega(t)||_{L^\infty(B(0,R_0-\epsilon))}\log \big(1+||
 \omega(t)||_{L^\infty(B(0,R_0-\epsilon))}\big)}.\]
This allows us to estimate the volume of the relevant super-level set of $|\omega|$ at time $t$,
 \begin{align*} \mbox{Vol}\bigg( \Lambda_t\bigg(\frac 1 {c_1} ||\omega(t)||_{L^\infty(B(0,R_0-\epsilon))} \bigg)\bigg) 
 &\leq \frac {c_1\log (c_1) ||\omega(t)||_{L^1(\{|\omega(x,t)|\geq
  \lambda(t)\}\cap B(0,R_0-\epsilon))}} {||\omega(t)||_{L^\infty(B(0,R_0-\epsilon))}\log \big(1+||
  \omega(t)||_{L^\infty(B(0,R_0-\epsilon))}\big)}
  \\& \quad+ \frac {c_1 ||\omega(t)\log(1+|\omega(t)|)||_{L^1(\{|\omega(x,t)|\geq
   \lambda(t)\}\cap B(0,R_0-\epsilon))} } {||\omega(t)||_{L^\infty(B(0,R_0-\epsilon))}\log \big(1+||
    \omega(t)||_{L^\infty(B(0,R_0-\epsilon))}\big)}
  \\&\leq \frac {K I_0(t) } {||\omega(t)||_{L^\infty(B(0,R_0-\epsilon))}\log \big(1+||
      \omega(t)||_{L^\infty(B(0,R_0-\epsilon))}\big)},
\end{align*}where we have set,
 \[I_0(t)=\int_{\{|\omega(x,t)|\geq
 \lambda(t)\}\cap B(0,R_0-\epsilon)}|\omega(x,t)| \log \big(1+|\omega(x,t)|\big)~dx,\] and have introduced a time-independent constant $K$  which depends on the fixed values $c_1$ and $R_0$ as well as the \emph{a priori} finite quantity $\sup_{0\leq t\leq T}||\omega(\cdot,t)||_{L^1(\{|\omega(x,t)|\geq
   \lambda(t)\}\cap B(0,R_0-\epsilon))}$. 

The second step ensures we can control $I_0(t)$ in terms of the finite bounds appearing in part (i) of the theorem. Tacitly summing over $j$, we have,
\begin{align*}I_0(t)&\leq C\int_{B(0,R_0-\epsilon)} |\omega_j|\log \big(1+\sqrt{\omega_1^2+\omega_2^2+\omega_3^2}\,\big)~dx.
\end{align*}An explicit reduction illustrates the argument (for simplicity we take $j=1$
and integrals over the indicated sets intersected with
$B(0,R_0-\epsilon)$),
\begin{align*}& \int |\omega_1|\log\big( 1+\sqrt{\omega_1^2+\omega_2^2+\omega_3^2}\,\big)~dx
\\&\leq \int_{\omega_1^2\leq \omega_2^2+\omega_3^2}  \sqrt{\omega_2^2+\omega_3^2}
\log\big( 1+ \sqrt{2( \omega_2^2+\omega_3^2)}\,\big)~dx+\int_{\omega_1^2> \omega_2^2+\omega_3^2}
|\omega_1| \log\big( 1+ \sqrt{2} |\omega_1|\big)~dx.
\end{align*}
Applying the same reasoning to the first integral above and then
repeating for all values of $j$ eventually yields,
\begin{align*}I_0(t)& \leq C\int w_i~dx +C\int w_i\log w_i~dx\leq C M_i.
\end{align*}
\end{proof}
 
The energy inequality-type construction used to prove the previous
theorem also works if we substitute $\log^m \omega_k$ (as defined in
Section \ref{sec:preliminaries}) in place of $\log w_k$. This allows
the application of amenable blow-up rates of order $1$ in
conjunction with Lemma \ref{lemma:uniformbmo}.c. To ensure things
are meaningful, we modify our definition of $q$ so that
$q(y)=\sqrt{e+y^2}$ (so, now, $w_k=\sqrt{e+\omega_k^2}$) and refer
to \eqref{decomp:F} to define $w_{k,+}$ and $w_{k,-}$.

\begin{theorem}\label{thrm:main2}Let $u$ be a Leray solution to
\eqref{eq:NSE} on $\mathbb{R}^3 \times (0,\infty)$, and suppose additionally that $\omega_0 = \nabla \times u_0 \in L^1 \cap L^2$. Denote by $T$ the first singular time and by $S$ the singular set of $\omega$ at time $T$. Fix positive values $R_0$ and $\epsilon$ so that $\emptyset\neq S\cap B(0,R_0)\subset B(0,R_0-\epsilon)$ -- i.e. there are singular points in $B(0,R_0)$ but all such points are in $B(0,R_0-\epsilon)$ -- and let  $Q_T(S)=Q_T(0,R_0,S)$.
\begin{itemize}\item[(i)] If $w_{k,+}|_{Q_T(S)}$ and $w_{k,-}|_{Q_T(S)}$ each exhibit amenable blow-up rates
of order $1$ in $Q_T(S)$, then there exists a positive value $M_k$ so
that, \begin{align*}\sup_{t\in [0,T]}
\int_{B(0,R_0-\epsilon)}|\omega_k(x,t )| \log \log
\big(\sqrt{e+|\omega_k(x,t)|^2}\big)~dx<M_k.
\end{align*}
\item[(ii)] Let $c_1$ ba a fixed constant which is greater than $1$. If, for all $k$, the premises of part (i) are satisfied, then there exists
a positive value $M_0$ so that, for $t\in [0,T]$,
 \begin{align*} \mbox{Vol}\bigg( \Lambda_t\bigg(\frac 1 {c_1}
 ||\omega(t)||_{L^\infty(B(0,R_0-\epsilon))}\bigg)\bigg) \leq \frac {M_0}
 {||\omega(t)||_{L^\infty(B(0,R_0-\epsilon))}\log \log \big(e+||\omega(t)||_{L^\infty(B(0,R_0-\epsilon))}
 \big)}. \end{align*}
\end{itemize}
\end{theorem}

\begin{proof}Multiplying equation \eqref{eq:vorticity} by
$ \psi q'(w_k)\big(\log\log w_k + w_k(\log\log)'w_k\big)$ we obtain
the evolution of $\psi w_k\log\log w_k$. The point-wise estimates in
\eqref{list:factsaboutlog} adapt directly to the function
$\log\log(y)$ (indeed, they adapt to any number of such
self-compositions of the logarithm) and, after integrating in space
and time, all of the estimates and cancellations from the previous
proof -- except those involving the vortex stretching term -- can be
duplicated directly.  We thus obtain, denoting by $R$ some {\em a
priori} finite quantity,
\begin{align}\label{ineq:raw2}\int (\psi w_k\log\log w_k)(t)~dx  \leq  \bigg|\int_0^t\int
\psi ~\omega\cdot \nabla u_k \frac {\omega_k} {w_k}\log\log w_k
~dx~ds\bigg|+R,\end{align} and, using
Lemma \ref{lemma:uniformbmo}.c., we are able to pull out of the
integral the uniformly-in-time bounded quantity,
\[\sup_t\bigg|\bigg|\psi\frac{\omega_k}{w_k}\log\log
w_k\bigg|\bigg|_{BMO},  \] and conclude exactly as in the previous
proof.\end{proof}

\begin{remark}
The two theorems apply to any local, spatially-algebraic blow-up scenario
in which the degrees of the polynomials stay \emph{uniformly bounded} near
the (possible) singular time. Geometrically, this corresponds to the number
of vortex filaments being uniformly bounded as the flow approaches the
singular time. This is somewhat unsatisfactory since, \emph{a priori}, one can
not rule out a scenario in which the number of coherent structures runs off
to infinity. The technical reason behind this restriction is that the bound on
$\| \log |P| \|_{BMO}$ blows up as the degree of $P$, $d$, goes to infinity.

Fortunately, the bound on $\| \log |P| \|_{BMO}$ is \emph{linear in $d$}
(cf. \cite{NaSoVo03}). Consequently, although the bound on the distribution
function, i.e., on the total volume of the super-level sets will blow
up (with $d$),
at least in the case of comparable volumes, the bound on the volume of a
single vortex filament will still be sub-critical (which suffices due to the
local nature of the argument).
\end{remark}

\subsection*{Acknowledgements.}

Z.B. acknowledges the support of the \emph{Virginia Space Grant
Consortium} via the Graduate Research Fellowship; Z.G. acknowledges
the support of the \emph{Research Council of Norway} via the grant
213474/F20, and the \emph{National Science Foundation} via the grant
DMS 1212023.

\bibliographystyle{plain}
\bibliography{references}
\end{document}